\documentclass[12pt,leqno]{amsart}
\usepackage{amssymb,amsfonts,amsmath,amsthm,amscd,mathrsfs}
\usepackage{psfrag}
\usepackage{color}
\def\IF{{\mathbb F}}

\def\IN{{\mathbb N}}

\def\n{\noindent}
\def\dis{\displaystyle}

\def\ov{\overline}
\def\wt{\widetilde}

\def\cal{\mathcal}

\def\cL{{\cal L}}
\def\cD{{\cal D}}

\def\cN{{\cal N}}

\def\cZ{{\cal Z}}

\def\cT{{\cal T}}

\def\Aut{{\rm Aut}}
\def\Im{{\rm Im}}
\def\pr{{\rm pr}}

\def\uH{\underline{H}}
\def\uG{\underline{G}}

\dimendef\dimen=0

\newtheorem{theorem}{Theorem}[section]
\newtheorem{lemma}[theorem]{Lemma}
\newtheorem{corollary}[theorem]{Corollary}
\newtheorem{proposition}[theorem]{Proposition}

\newtheorem{remark}[theorem]{Remark}
\newtheorem{example}[theorem]{Example}
\newtheorem{question}[theorem]{Question}

\begin{document}
\null
%\today
%\baselineskip14pt
\title[Lattices in products of trees and a theorem of H.C. Wang]{Lattices in products of trees and a theorem of H.C. Wang}
\author[M. Burger]{Marc Burger}
\address{M.B. \\ Department of Mathematics, ETH Zurich and The Institute for Advanced Study, Princeton.}
\author[S. Mozes]{Shahar Mozes}
\address{S.M. \\ Institute of Mathematics, Hebrew University of Jerusalem}
\thanks{M. B. thanks the Institute for Advanced Study where this work was completed for
its hospitality and support. S. M. thanks ETH Zurich for its hospitality and support as well as the support of  ISF grant 1003/11 and BSF grant 2010295}
\maketitle

\setcounter{section}{-1}
\section{Introduction}

In 1967 H.C. Wang \cite{Wa} showed that for a connected semisimple group $G$ without compact factors, a lattice $\Gamma < G$ is contained in only finitely many discrete subgroups. Shortly after, Kazhdan and Margulis \cite{KaMa} proved the stronger result that under the same hypothesis on $G$ there is a constant $c_G > 0$ such that ${\rm Vol}(\Gamma \backslash G) \ge c_G$ for every lattice $\Gamma < G$.

\medskip
In the meantime the scope of the study of lattices in locally compact groups has been greatly extended, including families of locally compact groups, like automorphism groups of trees \cite{BaLu}, products of trees \cite{BuMoZ}, and topological Kac-Moody groups \cite{Re}.

\medskip
In this context it was observed by Bass and Kulkarni that for the automorphism group of a $d$-regular tree $\cT_d$ the analog of Kazhdan-Margulis as well as Wang's theorem, fail.  They constructed for every $d \ge 3$ an infinite ascending chain $\Gamma_0 \lneq \Gamma_1 \lneq\cdots$ of discrete subgroups such that $\Gamma_\ell \backslash \cT_d$ is a geometric loop; when $d$ is composite they exhibited examples where $\Gamma_\ell \backslash \cT_d$ is a geometric edge \cite{BaKu}. On the other hand, when $d \ge 3$ is a prime number, a deep conjecture of Goldschmidt-Sims in finite group theory implies that there are, up to conjugacy, only finitely many discrete subgroups $\Gamma < {\rm Aut}\, \cT_d$ such that  $\Gamma \backslash \cT_d$ is an edge; this conjecture has been established for $d =3$ \cite{Go}. For a product $\cT_p \times \cT_q$ of trees of prime degrees, Y. Glasner \cite{Gl} proved the remarkable result that up to conjugacy there are only finitely many $\Gamma < {\rm Aut}\, \cT_p \times  {\rm Aut}\, \cT_q$ with non-discrete projections and such that $\Gamma \backslash (\cT_p \times \cT_q)$ is a geometric square.

\medskip
We propose to study this finiteness problem in the framework of the theory of lattices in products of trees developed in \cite{BuMo1},\,\cite{BuMo2}. Our aim is to establish Wang's theorem for co-compact lattices in a product ${\rm Aut}\, (T_1) \times {\rm Aut}\, (T_2)$ of automorphism groups of regular trees by imposing non-properness and certain local transitivity properties for their actions on the individual factors $T_1$ and $T_2$. 
The following follows from a more general result Theorem~\ref{thm:M} proved in section~\ref{sec:ProdTrees}

\begin{theorem}\label{theo0.1} Let $\Gamma <  {\rm Aut}\, T_1 \times {\rm Aut}\, T_2$ be a co-compact lattice and $G_i = \overline{{\rm pr}_i (\Gamma)}$ the closures of its projections.
%\medskip
Assume that each $G_i$ 
is non-discrete, locally quasiprimitive of constant, almost simple type.
%is vertex transitive, non-discrete and locally quasi-primitive of almost simple type. 
Then $\Gamma$ is contained in only finitely many discrete subgroups $\Lambda$ with $\Lambda < G_1 \times G_2$.
\end{theorem}

\medskip
Let us say that  a group $G$ acting on a tree is of constant type if for each vertex $x$, $\uG(x)$ the local permutation group induced by the action of $G(x)$,  the stabilizer of $x$, on $E(x)$, the set of edges based at $x$, is isomorphic to a fixed permutation group.
Recall that $G_i$ is called locally quasi-primitive if for every vertex $x$ of $T_i$ the permutation group $\uG_i(x)$ acting on on $E(x)$ is quasi-primitive. 
A quasi-primitive permutation group has one or two minimal normal subgroups and is called of almost simple type if there is a unique minimal normal subgroup which is simple non-abelian \cite{Pr}. The theorem applies for instance when $G_1$ and $G_2$ are locally $2$-transitive and the $2$-transitive permutation groups have non-abelian socle.

Under the weaker assumptions that $G_1$ and $G_2$ are locally quasi primitive we show in Proposition~\ref{prop2.2} that any
closed cocompact subgroup in $G_1 \times  G_2$ with dense projections is either discrete or open.
% subgroup containing $\Gamma$  is either discrete or has open closure.

\medskip
Theorem~\ref{theo0.1}  suggests the following question:

\begin{question}\label{que02} Given $G_i < {\rm Aut}\, T_i$ closed, non-discrete, locally quasi-primitive, does there exists a constant $C > 0$ such that ${\rm Vol}\big(\Gamma \backslash (G_1 \times G_2)\big) \ge C$ whenever $\Gamma < G_1 \times G_2$ is a co-compact lattice with dense projections.
\end{question}

\begin{example}\label{exam0.3} \rm Examples of lattices satisfying the hypothesis of Theorem \ref{theo0.1} can be found in Burger-Mozes \cite{BuMo2} and in Rattaggi \cite{Ra1},\,\cite{Ra2}. One constructs a finite square complex $Y$ with one vertex whose universal covering is $\cT_n \times \cT_m$. Then $\Gamma = \pi_1(Y)$ acts (simply) transitively on the vertices of $\cT_n \times \cT_m$. One obtains then up to permutation isomorphism one permutation group $P^m < S_m$, $P^n < S_n$ for each factor. Examples with non-discrete projections have been constructed for all alternating groups $(A_n, A_m)$ with $n,m$ even and $n \ge 30$, $m \ge 38$. Rattaggi has constructed many low degree examples among which $(A_6,M_{12})$, and $(A_6, A_6)$. All these are locally $2$-transitive transitive on both sides. An interesting example, also due to Rattaggi, of an irreducible lattice acting on $\cT_6 \times \cT_{10}$ realizes the pair $(A_6,S_5)$ where $A_6 < S_6$ and $S_5$ is endowed with its primitive action on the $10$ element set consisting of all $2$-element subsets of $\{1,2,3,4,5\}$.
\end{example}

\section{A general finiteness criterion}
\setcounter{equation}{0}

Given a lattice $\Gamma < G$ in a locally compact group $G$, a neighborhood $W$ of $e$ in $G$ and $\varepsilon > 0$, we define
\begin{align*}
R(\Gamma,W) & = \big\{\Lambda < G: \; \Lambda > \Gamma \;\mbox{and} \; \Lambda \cap W =\{e\}\big\}
\\[1ex]
\cL(\Gamma,\varepsilon) & = \big\{ \Lambda < G: \; \Lambda > \Gamma, \; \mbox{$\Lambda$ is discrete and ${\rm Vol}(\Lambda \backslash G) \ge \varepsilon$}\big\}.
\end{align*}

\n
Clearly, $R(\Gamma,W) \subset \cL(\Gamma, \varepsilon)$ where $\varepsilon = {\rm Vol}(\widetilde{W})$ and $\wt{W}$ is an open symmetric neighborhood of $e$ with $\wt{W}^2 \subset W$.

\begin{proposition}\label{prop1.1}
Let $\Gamma < G$ be a lattice. Assume that $\Gamma$ is finitely generated and that the centraliser $\cZ_G(\Gamma')$ of $\Gamma'$ in $G$ is discrete, for every subgroup $\Gamma' < \Gamma$ of finite index. Then $\cL(\Gamma,\varepsilon)$ is finite for every $\varepsilon > 0$. In particular, $R(\Gamma,W)$ is finite for every neighborhood $W \ni e$.
\end{proposition}

\begin{proof}
For every $\Lambda \in \cL(\Gamma,\varepsilon)$ let $K(\Lambda)$ be the kernel of the permutation action $\Lambda \rightarrow {\rm Sym}(\Gamma \backslash \Lambda)$. Clearly, we have the inclusions
\begin{equation}\label{1.1}
K(\Lambda) < \Gamma < \Lambda < \cN_G\big(K(\Lambda)\big)\,.
\end{equation}

\n
Since $[\Lambda:\Gamma] \le N : = \lfloor { \frac{1}\varepsilon}{\rm Vol} (\Gamma \backslash G) \rfloor$, we have the following estimate for the index of $K(\Lambda)$ in $\Gamma$:
\begin{equation}\label{1.2}
[\Gamma : K(\Lambda)] \le (N -1 )\,!
\end{equation}

\n
Since $\Gamma$ is finitely generated, the set $\Sigma_N$ of subgroups of index at most $(N - 1)!$ is finite; in addition, since $Z_G(K)$ is discrete, so is $\cN_G(K)$, and hence $\cN_G(K) / K$ is finite for every $K \in \Sigma_N$. It follows then from (\ref{1.1}) that
\begin{equation*}
\big| \cL(\Gamma, \varepsilon)\big| \le \dis\sum\limits_{K \in \Sigma_N} a(K)
\end{equation*}

\n
where $a(K)$ is the number of subgroups of the finite group $\cN_G(K)/K$. Thus $\cL(\Gamma, \varepsilon)$ is finite.
\end{proof}

\begin{corollary}\label{cor1.2}
Let $\Gamma < G = {\rm Aut} \,T_1 \times {\rm Aut}\,T_2$ be a co-compact lattice, where $T_1, T_2$ are regular trees. Then $\Gamma$ is finitely generated and $\cZ_G(\Gamma') = \{e\}$, for every $\Gamma'$ of finite index in $\Gamma$. In particular, given any neighborhood $W$ of $e$ in $G$, $R(\Gamma, W)$ is finite.
\end{corollary}

\begin{proof}
$\Gamma$ is a co-compact lattice in a compactly generated group and hence finitely generated. The assertion about the centralizer follows from the fact, applied to the projections ${\rm pr}_i(\Gamma')$, that the centralizer in ${\rm Aut}\,T_i$ of a group acting with finitely many orbits on the set of vertices of $T_i$ is trivial. The last assertion follows then from Proposition \ref{prop1.1}.
\end{proof}

\section{Products of trees}\label{sec:ProdTrees}
\setcounter{equation}{0}

Let $T = (X,Y)$ be a locally finite tree with vertex set $X$ and edge set $Y$; given a subgroup $H < {\rm Aut}\,T$, the stabilizer $H(x)$ of $x \in X$ induces on the set $E(x)$ of edges issued from $x$ a finite permutation group
\begin{equation*}
\uH(x) < {\rm Sym}\,E(x).
\end{equation*}

\n
We say that $H$ has locally a property $(P)$ if for every $x \in X$, the permutation group $\uH(x)$ has property $(P)$.

\medskip
Here we will be concerned with subgroups $H < {\rm Aut}\,T$ which are locally quasi-primitive. Recall that a finite permutation group $F < {\rm Sym} \,\Omega$ is quasi-primitive if every non-trivial normal subgroup of $F$ acts transitively on $\Omega$. There is a rich structure theory of quasi-primitive groups (see~\cite{Pr}) based on the study of their minimal normal subgroups. In this context $F$ is called almost simple if it has a unique minimal normal subgroup $M$ which is simple non-abelian; in this case $M < F < {\rm Aut}(M)$. Important examples of quasi-primitive permutation groups are given by the $2$-transitive permutation groups. In this context there is Burnside's theorem \cite{DiMo} which says that a $2$-transitive permutation group is either almost simple or the semi-direct product of a finite group with an $\IF_p$-vector space on which it acts linearly with precisely two orbits.

\medskip
We will need a basic result from the theory of locally quasi-primitive groups which we recall for the reader's convenience. Given a closed subgroup $H < {\rm Aut}\, T$ we let $H^{(\infty)}$ denote the intersection of all open subgroups of $H$ of finite index and $QZ(H) < H$ the subgroup consisting of all elements of $H$ whose centralizer is open in $H$. We have then (see~\cite{BuMo1}, Prop.~1.2.1):
\begin{theorem}\label{theo2.1}
Let $H < {\rm Aut} \,T$ be closed, non-discrete, locally quasi-primitive. Then
\begin{itemize}
\item[{\rm (1)}] $H/H^{(\infty)}$ is compact;
\item[{\rm (2)}] $QZ(H)$ acts freely on the set of vertices of $T$;
\item[{\rm (3)}] for any closed normal subgroup $N \triangleleft H$ we have either $N \subset QZ(H)$ or $N \supset H^{(\infty)}$.
\end{itemize}
\end{theorem}

As a warm-up we show a density result for non-discrete groups containing a lattice; such results are well known in the context of Lie groups where they are obtained as a consequence of the Borel density theorem.

\begin{proposition}\label{prop2.2}
For $i = 1, 2$, let $G_i < \Aut(T_i)$ be a non-discrete, locally quasi-primitive, closed subgroup acting cocompactly on $T_i$, such that $G_i^{(\infty)}$ is of finite index in $G_i$. Then any closed cocompact subgroup $L < G_1 \times G_2$ with dense projections on both $G_1$ and $G_2$, is either discrete or open; moreover, in the latter
case it contains $G_1^{(\infty)}\times G_2^{(\infty)}$.

%%%%Let $\Gamma <  {\rm Aut}\,T_1 \times {\rm Aut}\,T_2$ be a co-compact lattice such that $G_i : = \overline{{\rm pr}_i(\Gamma)}$ is non-discrete locally quasi-primitive, and $G_i^{(\infty)}$ is finite index in $G_i$. Let $\Lambda < G_1 \times G_2$ be a subgroup containing $\Gamma$. Then either $\Lambda$ is discrete or its closure $\overline{\Lambda}$ contains $G_1^{(\infty)} \times G_2^{(\infty)}$, in particular it is open in $G_1 \times G_2$.
\end{proposition}

\medskip\n
{\bf Notation:} Here and in the sequel, for a tree $T = (X,Y)$, $d_T$ denotes the combinatorial distance on $X$, $B_T(x,\ell)$ the ball of radius $\ell \in \IN$ with center $x$ and, if $H < {\rm Aut}\,T$, $H_\ell(x) = \{g \in H$: $g$ is the identity on $B_T(x,\ell)\}$.

\begin{proof}
Assume that L is not discrete.

\medskip\n
{\bf Claim:} If $L \cap (\{e\} \times G_2)$ doesn't act freely on $X_2$, the set of vertices of $T_2$, then $L \supset G_1^{(\infty)} \times G_2^{(\infty)}$.

\medskip
Indeed, $L \cap (\{e\} \times G_2)$ is, with a slight abuse of notation, a closed subgroup of $G_2$ which is normal in ${\rm pr}_2(L)$, hence in $G_2$ since $\ov{{\rm pr}_2(L)} = G_2$; since it does not act freely on $X_2$ it must by Theorem \ref{theo2.1} contain $\{e\} \times G_2^{(\infty)}$. Consider the continuous injective map
\begin{equation*}
j: (G_1 \times \{e\})/((G_1 \times \{e\}) \cap L) \hookrightarrow (G_1 \times G_2) / L\,.
\end{equation*}

\n
Its image, lifted to $G_1 \times G_2$ is $\big(G_1 \times \{e\}\big) \cdot L \supset G_1 \times G_2^{(\infty)}$ and hence is open, therefore closed. Thus the image of $j$ is closed; since source and target of $j$ are locally compact metrisable we conclude that $j$ is a homeomorphism onto its image, and since $\Im(j) \subset (G_1 \times G_2)/L$ is therefore compact we deduce that $G_1/(G_1 \times \{e\}) \cap L$ is compact.

Thus $(G_1 \times \{ e \}) \cap \ L$ is a cocompact normal subgroup of $G_1$. If it does not
contain $G_1^{(\infty)}$, then by Theorem~\ref{theo2.1} it is discrete and contained in $QZ(G_1)$.
Since it is cocompact, it is finitely generated. Any finitely generated subgroup of the quasi-center has an open centraliser; but the centraliser of a group acting cocompactly on $T_1$ must be trivial. This implies that $G_1$ is discrete, a contradiction. Hence $G_1^{(\infty)} \subset(G_1 \times \{ e \}) \cap L$. Arguing similarly for $G_2^{(\infty)}$ completes the proof of the claim.

%	%	We claim that it follows that the action of $L\cap(G_1\times \{e\})$ on $T_1$ is not free. Indeed we may choose a finite subset $S\subset L\cap(G_1\times \{e\})$ such that $\langle S \rangle\backslash T_1$ is finite. This implies that the centralizer of $S$ in $G_1$ is trivial. It now follows using the fact that the projection of $L$ to $G_1$ is dense that some commutator of an element of $s\in S$ and a element of $L$ with sufficiently small projection in $G_1$ is non trivial and fixes some vertex in $T_1$. Applying again as above Theorem~\ref{theo2.1} we deduce that $L\supset (G_1^{(\infty)}\times \{e\})$, proving the claim.

%	
% But by a similar use of Theorem \ref{theo2.1} we deduce that $L \supset G_1^{(\infty)}$; this proves the claim.

\medskip
Let now $(u,v) \in X_1 \times X_2$ and consider for every $\ell \ge 1$ the intersection $L \cap G_1(u) \times G_{2,\ell}(v)$. Since $L$ is non-discrete this intersection is non-trivial $\forall \ell \ge 1$. We distinguish two cases:
\begin{itemize}
\item[(1)] $L \cap \big(G_1(u) \times G_{2,\ell}(v)\big) \subset \{e\} \times G_{2,\ell}(v)$
for some $\ell \ge 1$.

\item[(2)] $L \cap \big(G_1(u) \times G_{2,\ell}(v)\big)$ has for every $\ell \ge 1$ a non-trivial projection on the first factor.
\end{itemize}

\medskip\n
In the first case we have that $L \cap (\{e\} \times G_2)$ does not act freely on $X_2$ and hence by the claim $L \supset G_1^{(\infty)} \times G_2^{(\infty)}$; but since $G_1^{(\infty)}$ is non-discrete as well, we have that $L \cap \big(G_1(u) \times G_{2,\ell}(v)\big)$ has non-trivial first projection, so this case does not occur.

\medskip
In order to analyze  the second case let $\cD_1 \times \cD_2 \subset X_1 \times X_2$ be a finite set with $\bigcup_{\gamma \in \Gamma} \gamma(\cD_1 \times \cD_2) = X_1 \times X_2$.

\medskip
For every $\ell \ge 1$, let $u_\ell \in X_1$ be such that the projection of $L \cap \big(G_1(u_\ell) \times G_{2,\ell}(v)\big)$ into $\uG_1(u_\ell)$ is non-trivial; such a $u_\ell$ can be obtained by considering the subtree $T_1^\ell$ spanned by the ${pr}_1\big(L \cap G_1(u) \times G_{2,\ell}(v)\big)$-fixed vertices in $X_1$ and taking for $u_\ell$ a vertex in $T_1^\ell$ which has a neighbor which is not in $T_1^\ell$.

\medskip
Use now the $\Gamma$-action to find $(w_\ell,v_\ell) \in (\cD_1 \times \cD_2) \cap \Gamma \cdot(u_\ell,v)$ with the property that $L \cap \big(G_1(w_\ell) \times G_{2,\ell}(v_\ell)\big)$ projects non-trivially to $\uG_1(w_\ell)$.

\medskip
Since $\cD_1 \times \cD_2$ is finite and $\uG_1(w)$ is finite we may assume, by passing to an appropriate subsequence, that $(w_\ell,v_\ell) = (w_0,v_0)$ and that there is $(a_\ell,b_\ell) \in L \cap \big(G_1(w_0) \times G_{2,\ell}(v_0)\big)$ such that $a_\ell$ induces on $E(w_0)$ a fixed non-trivial permutation. Passing to subsequence again we may assume that $(a_\ell,b_\ell)$ converges; its limit must necessarily have the form $(a,e)$ where $a \in G_1(w_0)$, $a \not= e$. This shows that $L \cap (G_1(w_0) \times \{e\})$ is non-trivial; hence $L \cap (G_1 \times \{e\})$ acts non-freely on $X_1$ and by the claim applied to $G_1$ instead of $G_2$ we deduce $L \supset G_1^{(\infty)} \times G_2^{(\infty)}$.
\end{proof}

%The main step in the proof of Theorem \ref{theo0.1} in the introduction is the following lemma:

%%pages 1.1-1.2 of notes
%%%%%%%%%%%%
We turn now to our main theme and prove a general result implying Theorem~\ref{theo0.1} of the introduction.
This result which is of a more technical nature requires some preliminary comments.
A quasi-primitive permutation group has either a unique minimal normal subgroup or two minimal normal subgroups (cf. \cite{Pr}).
In the latter case  each of them acts regularly and they are isomorphic. Thus in both cases the isomorphism class of a simple factor of a minimal normal subgroup of a quasi-primitive permutation group is well defined.

Let $G_i<\Aut T_i$ be of constant type and locally quasiprimitive. We will denote by $M_i$ a simple factor of a minimal normal subgroup of $\uG_i(x)$ and by $S_i$ the stabilizer in $\uG_i(x)$ of a typical element in $E(x)$.

\begin{theorem}\label{thm:M}
Let $G_i<\Aut T_i$ be a closed, non-discrete, constant type and locally quasi-primitive group. Let $\Gamma< G_1 \times G_2$ be a cocompact lattice with dense projections. Then there is a neighborhood $W$ of $(e,e)$ in $G_1\times G_2$ with the following property: If there is a discrete subgroup $\Lambda < G_1 \times G_2$ containing $\Gamma$ with $\Lambda \cap W \neq \{ (e,e)\}$ then $M_1$ is a section of $S_2$ and $M_2$ is a section of $S_1$.
\end{theorem}

\begin{lemma}\label{lem2.3}
Let $\Gamma < {\rm Aut} \,T_1 \times {\rm Aut}\,T_2$ be a co-compact lattice such that the closure of the projections  $G_i : = \ov{{\rm pr}_i(\Gamma)}$ are non-discrete, locally quasi-primitive. Given $\ell \ge 1$, there is $W_1$ neighborhood of $(e,e)$ in $G_1 \times G_2$ depending on $\ell$ and $\Gamma$ with the property that if $\Lambda$ is a discrete subgroup of $G_1 \times G_2$ containing $\Gamma$, with $\Lambda \cap W_1 \not= (e,e)$, then there is $(w,v) \in X_1 \times X_2$ such that the image of $\Lambda \cap \big(G_1(w) \times G_{2,\ell}(v)\big)$ in $\uG_1(w)$ contains a non-trivial normal subgroup of $\uG_1(w)$.
\end{lemma}

\noindent
\begin{remark}\label{rem:inject}
In the proof of Lemma~\ref{lem2.3} and of Theorem~\ref{thm:M} we will use the following fact: If $F<\Lambda\cap (G_1(a)\times G_2(b))$ is a subgroup then $F$ injects into each of the two factors. Indeed, for instance, the intersection $F\cap(\{e\} \times G_2)$ is contained in the discrete subgroup $\Lambda\cap(\{e\} \times G_2)$ which is normal in $\overline{\pr_2(\Lambda)}=G_2$ and hence must be contained in $\{e\} \times QZ(G_2)$ by theorem 2.1. But $QZ(G_2)$ acts freely in $X_2$ and $F$ fixes $(a,b)$ which implies $F\cap (\{e\} \times G_2) = \{(e,e)\}$
\end{remark}

\begin{proof} {\em (of Lemma 2.4.)}
Since ${\rm pr}_1(\Gamma)$ is dense in $G_1$, the projection onto $G_1(x)$ of $\Gamma \cap (G_1(x) \times G_2)$ is dense and hence we can choose $\gamma^x_1,\gamma^x_2,\dots,\gamma^x_r$ in $\Gamma \cap (G_1(x) \times G_2)$ which are representatives of the elements of the finite group $\uG_1(x)$; observe that by repeating elements we can take $r$ independent of $x$. Let $\cD_1 \times \cD_2 \subset X_1 \times X_2$ be a finite set such that $\bigcup_{\gamma \in \Gamma} \gamma(\cD_1 \times \cD_2) = X_1 \times X_2$ and define
\begin{equation*}
R : = \max\limits_{1 \le i \le r} \; \max\limits_{(u,v) \in \cD_1 \times \cD_2} \,d_{T_2} (\psi^u_i (v),v)
\end{equation*}
where $\gamma_i^u = (\varphi^u_i,\psi^u_i)$.

\medskip
Fix $(u,v) \in \cD_1 \times \cD_2$, then we claim that $W_1 : = G_1(u) \times G_{2,R + \ell}(v)$ is the neighborhood announced in the lemma. For ease of notation denote for every $n \ge 0$, $(a,b) \in X_1 \times X_2$ by $F^n_{(a,b)}$ the intersection $\Lambda \cap G_1(a) \times G_{2,n}(b)$. 
By Remark~\ref{rem:inject} we observe first that $F^n_{(a,b)} \cap ( \{e\} \times G_2) = (e,e)$.
% indeed this intersection is contained in the discrete subgroup $\Lambda \cap (\{e\} \times G_2)$ which is normal in $\ov{{\rm pr}_2(\Lambda)} = G_2$ and hence must be contained in $\{e\} \times QZ(G_2)$ by Theorem \ref{theo2.1}; but $QZ(G_2)$ acts freely on $X_2$ and $F^n_{a,b}$ fixes $b \in X_2$ which implies that $F^n_{a,b} \cap (\{e\} \times G_2) = (e,e)$.

\medskip
Assume now that $F_{(u,v)}^{R + \ell} = \Lambda \cap W_1 \not= (e,e)$. According to the preceding observation, we know that $F_{(u,v)}^{R+ \ell}$ injects into $G_1(u)$; since this group is non-trivial we can by the same argument as in Proposition \ref{prop2.2} find a vertex $w \in X_1$ which is $F^{R + \ell}_{(u,v)}$-fixed and such that  the image of $F_{(u,v)}^{R + \ell}$ in $\uG_1(w)$ is not trivial; as a result $F_{(w,v)}^{R + \ell}$ has non-trivial image in $\uG_1(w)$. Using the $\Gamma$-action we may assume that $(w,v) \in \cD_1 \times \cD_2$. Next we claim that for all $1 \le i \le r$:
\begin{equation*}
(\gamma_i^w)^{-1} \,F^{R+ \ell}_{(w,v)} \;\gamma^w_i \subset F_{(w,v)}^\ell .
\end{equation*}

\n
Let $(g,h) \in F_{(w,v)}^{R + \ell} = \Lambda \cap \big(G_1(w) \times G_{2, R+ \ell} (v)\big)$. Then $(\gamma_i^w)^{-1} (g,h) \;\gamma_i^w = \big((\varphi_i^w)^{-1}\, g  \varphi_i^w, \,(\psi_i^w)^{-1}\, h \psi_i^w\big)$ and $(\varphi_i^w)^{-1} g\,\varphi_i^w \in G_1(w)$ since all terms involved are in $G_1(w)$. For the second component let $x \in B_{T_2}(v,\ell)$, then
\begin{equation*}
\begin{split}
d_{T_2} (\psi^w_i(x),v) & \le d\big(\psi_i^w(x), \, \psi_i^w(v)\big) +d\big(\psi^w_i(v),v\big)
\\
& \le \ell + R
\end{split}
\end{equation*}
and thus $h \psi_i^w(x) = \psi_i^w(x)$ or $(\psi_i^w)^{-1} h\,\psi_i^w(x) = x$. This shows the claim.

\medskip
Thus the image of $F_{(w,v)}^\ell$ in $\uG_1(w)$ contains all conjugates of some non-trivial element and hence contains a non-trivial normal subgroup of $\uG_1(w)$.
\end{proof}

\medskip\n
%%Pages 3.1-3.3 of notes
%%%%%%%%%%%%
{\bf Proof of Theorem~\ref{thm:M}.}
Let $W_1$, $(w,v)\in X_1 \times X_2$ be given by lemma~\ref{lem2.3} where we take $\ell=1$. In the notation of the proof of Lemma~\ref{lem2.3} we have that the image of $F_{(w,v)}^1$ in $\uG_1(w)$ contains a non-trivial normal subgroup and hence $M_1$.
Thus the simple group $M_1$ is a section of $F_{(w,v)}^1$. By Remark~\ref{rem:inject} the group $F_{(w,v)}^1$ injects into $G_{2,1}(v)$. For every $j\ge 1$, the quotient $G_{2,j}(v)/G_{2,j+1}(v)$ is isomorphic to  a subgroup of $S_2^{a(j)}$ where $a(j)$ is the cardinality of the sphere of radius $j$ in $T_2$. Thus $F_{(w,v)}^1$ admits a Jordan-H\"older series consisting of subgroups of $S_2$ and hence $M_1$ is a section of $S_2$.
\qed

\medskip\noindent
{\bf Proof of Theorem~\ref{theo0.1}.}
Let $W$ be the neighborhood of $(e,e)\in G_1 \times G_2$ given by Theorem~\ref{thm:M} and assume that $\Lambda< G_1 \times G_2$ is a discrete subgroup containing $\Gamma$ with $\Lambda \cap W \neq \{(e,e)\}$.
But now $M_i$ coincides with the unique minimal normal subgroup of $\uG_i(x)$, in particular it acts transitively; thus $|M_i \cap S_i | < | M_i |$, since $|M_i/(M_i \cap S_i)|$ is the degree of the tree $T_i$. Since the centralizer of $M_i$ in $\uG_i(x)$ is trivial, the group $S_i/(S_i \cap M_i)$ is isomorphic to  subgroup of ${\rm Out}(M_i)$ which by the Schreier conjecture  (which asserts that the group of outer automorphisms of a finite simple group is solvable) (cf. {\cite{DiMo}, page 133}) is solvable.
Since $M_1$ is a section of $S_2$ this implies that it is a section of $S_2 \cap M_2$ and hence 
$|M_1|\le |S_2 \cap M_2|<|M_2|$.
By the same argument interchanging between the indices $1$ and $2$ we get also $|M_2| < |M_1|$ and thus a contradiction.
\qed

\medskip
\noindent
{\bf Acknowledgment.}
The authors thank the referee for the careful reading of the manuscript and the useful suggestions for improvements.

\end{document}